\documentclass[9pt,twoside,reqno]{amsart}
\usepackage{amssymb}
\usepackage{amsmath}
\usepackage{amsfonts}

\newtheorem{theorem}{Theorem}
\theoremstyle{plain}

\newtheorem{corollary}{Corollary}
\newtheorem{definition}{Definition}

\newtheorem{proposition}{Proposition}
\newtheorem{remark}{Remark}

\numberwithin{equation}{section}
\input{tcilatex}

\begin{document}
\title[Inequalities Related to $h-$logaritmic, $h-$geometric and $h-$multi
Convex Functions]{DEFINITIONS OF $h-$LOGARITMIC, $h-$GEOMETRIC AND $h-$MULTI
CONVEX FUNCTIONS AND SOME INEQUALITIES RELATED TO THEM}
\author{M.EM\.{I}N \"{O}ZDEM\.{I}R$^{\bigstar }$}
\address{$^{\bigstar }$ATAT\"{U}RK UNIVERSITY, K. K. EDUCATION FACULTY,
DEPARTMENT OF MATHEMATICS, 25240, CAMPUS, ERZURUM, TURKEY}
\email{emos@atauni.edu.tr}
\author{MEVL\"{U}T TUN\c{C}$^{\blacksquare }$}
\address{$^{\blacksquare }$UNIVERSITY OF K\.{I}L\.{I}S 7 ARALIK, FACULTY OF
ARTS AND SCIENCES, DEPARTMENT OF MATHEMATICS, 79000, K\.{I}L\.{I}S, TURKEY}
\email{mevluttunc@kilis.edu.tr}
\author{MUSTAFA G\"{U}RB\"{U}Z$^{\blacktriangle }$}
\address{$^{\blacktriangle }$ UNIVERSITY OF A\u{G}RI \.{I}BRAH\.{I}M \c{C}E%
\c{C}EN, EDUCATION FACULTY, DEPARTMENT OF MATHEMATICS, 04100, A\u{G}RI,
TURKEY}
\email{mgurbuz@agri.edu.tr}
\thanks{$^{\blacktriangle }$Corresponding Author.}

\begin{abstract}
In this paper, we put forward some new definitions and integral inequalities
by using fairly elementary analysis.
\end{abstract}

\maketitle

\section{\textbf{Introduction}}

The following inequality is well known in the literature as the
Hermite-Hadamard integral inequality \cite{hdm}.

\begin{equation}
f\left( \frac{a+b}{2}\right) \leq \frac{1}{b-a}\int_{a}^{b}f\left( x\right)
dx\leq \frac{f\left( a\right) +f\left( b\right) }{2}  \label{101}
\end{equation}%
where $f:I\subset 
\mathbb{R}
\rightarrow 
\mathbb{R}
$ is a convex function on the interval $I$ of real numbers and $a,b\in I$
with $a<b.$

The following definitions is well known in the literature.

\begin{definition}
\bigskip A function $f:I\rightarrow 
\mathbb{R}
,$ $\emptyset \neq I\subseteq 
\mathbb{R}
,$ where $I$ is a convex set, is said to be convex on $I$ if inequality%
\begin{equation}
f\left( tx+\left( 1-t\right) y\right) \leq tf\left( x\right) +\left(
1-t\right) f\left( y\right)
\end{equation}%
holds for all $x,y\in I$ and $t\in \left[ 0,1\right] $.
\end{definition}

The concept of $h$-convexity was introduced by Varo\v{s}anec \cite{var} and
was generalized by H\'{a}zy \cite{H11}.

\begin{definition}
\bigskip \cite{var} \textit{Let }$h:J\rightarrow 
\mathbb{R}
$\textit{\ be a non-negative function, }$h\not\equiv 0$\textit{\ . We say
that }$f:I\rightarrow 
\mathbb{R}
$\textit{\ is an }$h$-\textit{convex function, or that }$f$\textit{\ belongs
to the class }$SX\left( h,I\right) $\textit{, if }$f$\textit{\ is
non-negative and for all }$x,y\in I$\textit{\ and }$t\in \left( 0,1\right) $%
\textit{\ we have \ \ \ \ \ \ \ \ \ \ \ \ }%
\begin{equation}
f\left( tx+\left( 1-t\right) y\right) \leq h\left( t\right) f\left( x\right)
+h\left( 1-t\right) f\left( y\right) .  \label{106}
\end{equation}
\end{definition}

\bigskip If inequality (\ref{106}) is reversed, then $f$ is said to be $h$%
-concave, i.e. $f\in SV\left( h,I\right) $. Obviously, if $h\left( t\right)
=t$, then all nonnegative convex functions belong to $SX\left( h,I\right) $\
and all nonnegative concave functions belong to $SV\left( h,I\right) $; if $%
h\left( t\right) =\frac{1}{t}$, then $SX\left( h,I\right) =Q\left( I\right) $%
; if $h\left( t\right) =1$, then $SX\left( h,I\right) \supseteq P\left(
I\right) $; and if $h\left( t\right) =t^{s}$, where $s\in \left( 0,1\right) $%
, then $SX\left( h,I\right) \supseteq K_{s}^{2}$.

\begin{definition}
\cite{alzer} A function $h:J\rightarrow 
\mathbb{R}
$ is said to be a superadditive function if%
\begin{equation}
h\left( x+y\right) \geq h\left( x\right) +h\left( y\right)  \label{108}
\end{equation}%
for all $x,y\in J.$
\end{definition}

Recently, In \cite{zh}, the concept of geometrically and $s$-geometrically
convex functions was introduced as follows.

\begin{definition}
\cite{zh} A function $f:I\subset 
\mathbb{R}
_{+}\rightarrow 
\mathbb{R}
_{+}$\textit{\ }is said to be a geometrically convex function if%
\begin{equation}
f\left( x^{t}y^{1-t}\right) \leq \left[ f\left( x\right) \right] ^{t}\left[
f\left( y\right) \right] ^{1-t}  \label{d3}
\end{equation}%
for \textit{all }$x,y\in I$\textit{\ and }$t\in \left[ 0,1\right] $\textit{.}
\end{definition}

\begin{definition}
\cite{zh} A function $f:I\subset 
\mathbb{R}
_{+}\rightarrow 
\mathbb{R}
_{+}$\textit{\ }is said to be an $s$-geometrically convex function if%
\begin{equation}
f\left( x^{t}y^{1-t}\right) \leq \left[ f\left( x\right) \right] ^{t^{s}}%
\left[ f\left( y\right) \right] ^{\left( 1-t\right) ^{s}}  \label{d4}
\end{equation}%
for some $s\in \left( 0,1\right] $, $x,y\in I$\textit{\ and }$t\in \left[ 0,1%
\right] $\textit{.}
\end{definition}

If $s=1$, the $s$-geometrically convex function becomes a geometrically
convex function on $%
\mathbb{R}
_{+}$.

\bigskip In \cite{at}, Tun\c{c} and Akdemir introduced the class of $s$%
-logarithmically convex functions in the first and second sense as the
following:

\begin{definition}
\label{mm}\textit{A function }$f:I\subset 
\mathbb{R}
_{0}\rightarrow 
\mathbb{R}
_{+}$\textit{\ is said to be }$s$-logarithmically convex in the first sense%
\textit{\ if \ \ \ \ \ \ \ \ \ \ \ \ }%
\begin{equation}
f\left( \alpha x+\beta y\right) \leq \left[ f\left( x\right) \right]
^{\alpha ^{s}}\left[ f\left( y\right) \right] ^{\beta ^{s}}  \label{m1}
\end{equation}%
for some $s\in \left( 0,1\right] $, where $x,y\in I$\textit{\ and }$\alpha
^{s}+\beta ^{s}=1.$
\end{definition}

\begin{definition}
\label{mmm}\textit{A function }$f:I\subset 
\mathbb{R}
_{0}\rightarrow 
\mathbb{R}
_{+}$\textit{\ is said to be }$s$-logarithmically convex in the second sense%
\textit{\ if \ \ \ \ \ \ \ \ \ \ \ \ }%
\begin{equation}
f\left( tx+\left( 1-t\right) y\right) \leq \left[ f\left( x\right) \right]
^{t^{s}}\left[ f\left( y\right) \right] ^{\left( 1-t\right) ^{s}}  \label{m2}
\end{equation}%
for some $s\in \left( 0,1\right] $, where $x,y\in I$\textit{\ and }$t\in %
\left[ 0,1\right] $\textit{.}
\end{definition}

\bigskip Clearly, when taking $s=1$ in Definition \ref{mm} or Definition \ref%
{mmm}, then $f$ becomes the standard logarithmically convex function on $I$.

\section{\textbf{Results for }$h$-$\log $-\textbf{Convex Functions}}

\begin{definition}
\label{d1}A positive function $f$ is called $h$-logarithmically convex on a
real interval $I=\left[ a,b\right] ,$ if for all $x,y\in I$ and $t\in \left[
0,1\right] ,$%
\begin{equation}
f\left( tx+\left( 1-t\right) y\right) \leq \left[ f\left( x\right) \right]
^{h\left( t\right) }\left[ f\left( y\right) \right] ^{h\left( 1-t\right) }
\label{d11}
\end{equation}%
where $h\left( t\right) $\ is a nonnegative function on $J,$ with $%
h:J\subseteq 
\mathbb{R}
\rightarrow 
\mathbb{R}
.$
\end{definition}

If $f$ is a positive $h$-logarithmically concave function, then inequality
is reversed. On the other hand, a function $f$ is $h$-logarithmically convex
on $I$ if $f$ is positive and $\log f$ is $h$-convex on $I$.

\begin{proof}
Let's rewrite $g=\log f\left( x\right) .$ Since $g$ is $h-$convex function,
for $x,y\in I$ and $t\in \left[ 0,1\right] $ we get%
\begin{eqnarray*}
g\left( tx+\left( 1-t\right) y\right) &\leq &h\left( t\right) g\left(
x\right) +h\left( 1-t\right) g\left( y\right) \\
\log f\left( tx+\left( 1-t\right) y\right) &\leq &h\left( t\right) \log
f\left( x\right) +h\left( 1-t\right) \log f\left( y\right) \\
&=&\log f\left( x\right) ^{h\left( t\right) }+\log f\left( y\right)
^{h\left( 1-t\right) }
\end{eqnarray*}%
So we have%
\begin{eqnarray*}
f\left( tx+\left( 1-t\right) y\right) &\leq &e^{\log f\left( x\right)
^{h\left( t\right) }}.e^{\log f\left( y\right) ^{h\left( 1-t\right) }} \\
&=&\left[ f\left( x\right) \right] ^{h\left( t\right) }\left[ f\left(
y\right) \right] ^{h\left( 1-t\right) }.
\end{eqnarray*}
\end{proof}

\begin{remark}
If we take $h\left( t\right) =t$ in Definition \ref{d1}, $h$-logarithmically
convex (concave) become ordinary $\log $-convex (concave) function, and if
we take $h\left( t\right) =t^{s}$ in Definition \ref{d1}, $h$%
-logarithmically convex (concave) become $s$-$\log $-convex (concave)
function in the second sense.
\end{remark}

\begin{proposition}
Let $f$ be an $h$-$\log $-convex function. If the function $h$ satisfies the
condition%
\begin{equation*}
h\left( t\right) +h\left( 1-t\right) =1
\end{equation*}%
for all $t\in \left[ 0,1\right] ,$ then $f$ is also $h$-convex function.
\end{proposition}

\begin{proof}
As we choose $f$ is $h$-$\log $-convex function we can write%
\begin{equation*}
f\left( tx+\left( 1-t\right) y\right) \leq \left[ f\left( x\right) \right]
^{h\left( t\right) }\left[ f\left( y\right) \right] ^{h\left( 1-t\right) }.
\end{equation*}%
From a simple inequality%
\begin{equation*}
x^{\alpha }y^{1-\alpha }\leq \alpha x+\left( 1-\alpha \right) y
\end{equation*}%
for $x,y>0$ and by using the condition $h\left( t\right) +h\left( 1-t\right)
=1$ we have%
\begin{eqnarray*}
f\left( tx+\left( 1-t\right) y\right) &\leq &\left[ f\left( x\right) \right]
^{h\left( t\right) }\left[ f\left( y\right) \right] ^{h\left( 1-t\right) } \\
&\leq &h\left( t\right) f\left( x\right) +h\left( 1-t\right) f\left( y\right)
\end{eqnarray*}%
which shows that $f$ is $h$-convex function.
\end{proof}

\begin{theorem}
\label{10}Let $f$ be an $h$-$\log $-convex function. If $f$ is monotonically
increasing or decreasing and $h$ is superadditive function on $\left[ 0,1%
\right] $, we have%
\begin{equation}
\frac{1}{b-a}\int_{a}^{b}f\left( x\right) f\left( a+b-x\right) dx\leq \left[
f\left( a\right) f\left( b\right) \right] ^{h\left( 1\right) }  \label{a}
\end{equation}%
and%
\begin{equation}
\frac{1}{\left( b-a\right) ^{2}}\left( \int_{a}^{b}f\left( x\right)
dx\right) ^{2}\leq \left[ f\left( a\right) f\left( b\right) \right]
^{h\left( 1\right) }.  \label{b}
\end{equation}
\end{theorem}

\begin{proof}
Since $f$ is an $h$-$\log $-convex function, for $a,b\in I,$ $t\in \left[ 0,1%
\right] $ we have%
\begin{equation*}
f\left( ta+\left( 1-t\right) b\right) \leq \left[ f\left( a\right) \right]
^{h\left( t\right) }\left[ f\left( b\right) \right] ^{h\left( 1-t\right) }
\end{equation*}%
and%
\begin{equation*}
f\left( tb+\left( 1-t\right) a\right) \leq \left[ f\left( b\right) \right]
^{h\left( t\right) }\left[ f\left( a\right) \right] ^{h\left( 1-t\right) }.
\end{equation*}%
If we multiply both sides we have%
\begin{eqnarray*}
f\left( ta+\left( 1-t\right) b\right) f\left( tb+\left( 1-t\right) a\right)
&\leq &\left[ f\left( a\right) f\left( b\right) \right] ^{h\left( t\right) }%
\left[ f\left( a\right) f\left( b\right) \right] ^{h\left( 1-t\right) } \\
&=&\left[ f\left( a\right) f\left( b\right) \right] ^{h\left( t\right)
+h\left( 1-t\right) }.
\end{eqnarray*}%
As we choose $h$ is superadditive function we get%
\begin{equation*}
f\left( ta+\left( 1-t\right) b\right) f\left( tb+\left( 1-t\right) a\right)
\leq \left[ f\left( a\right) f\left( b\right) \right] ^{h\left( 1\right) }.
\end{equation*}%
By integrating the last inequality over $t$ from $0$ to $1$ we get%
\begin{equation*}
\frac{1}{b-a}\int_{a}^{b}f\left( x\right) f\left( a+b-x\right) dx\leq \left[
f\left( a\right) f\left( b\right) \right] ^{h\left( 1\right) }.
\end{equation*}%
So the proof of (\ref{a}) is completed.

On the other hand is we use Chebyshev inequality on (\ref{a}) we have%
\begin{eqnarray*}
\frac{1}{b-a}\int_{a}^{b}f\left( x\right) f\left( a+b-x\right) dx &\geq &%
\frac{1}{\left( b-a\right) ^{2}}\int_{a}^{b}f\left( x\right)
dx\int_{a}^{b}f\left( a+b-x\right) dx \\
&=&\frac{1}{\left( b-a\right) ^{2}}\left( \int_{a}^{b}f\left( x\right)
dx\right) ^{2}.
\end{eqnarray*}%
So we have%
\begin{equation*}
\frac{1}{\left( b-a\right) ^{2}}\left( \int_{a}^{b}f\left( x\right)
dx\right) ^{2}\leq \left[ f\left( a\right) f\left( b\right) \right]
^{h\left( 1\right) }.
\end{equation*}%
Then proof of (\ref{b}) is completed.
\end{proof}

\begin{corollary}
If we choose $h\left( t\right) =\frac{1}{t}$ at Theorem \ref{10} as a
superadditive function on $\left[ 0,1\right] $\ we have%
\begin{equation*}
\frac{1}{b-a}\int_{a}^{b}f\left( x\right) f\left( a+b-x\right) dx\leq
f\left( a\right) f\left( b\right)
\end{equation*}%
and%
\begin{equation*}
\frac{1}{\left( b-a\right) ^{2}}\left( \int_{a}^{b}f\left( x\right)
dx\right) ^{2}\leq f\left( a\right) f\left( b\right) .
\end{equation*}
\end{corollary}

\begin{theorem}
Let $f$ and $g$ are $h$-$\log $-convex functions on $I$ and let $h$ is
symmetric about $\frac{1}{2}.$ For $a,b\in I$ and $t\in \left[ 0,1\right] $
we have%
\begin{equation}
\frac{1}{b-a}\int_{a}^{b}\left( fg\right) \left( x\right) dx\leq \int_{0}^{1}%
\left[ \left( fg\right) \left( a\right) \left( fg\right) \left( b\right) %
\right] ^{h\left( t\right) }dt.  \label{c}
\end{equation}
\end{theorem}

\begin{proof}
As we choose $f$ and $g$ are $h$-$\log $-convex functions on $I$ we have%
\begin{eqnarray*}
f\left( ta+\left( 1-t\right) b\right) &\leq &\left[ f\left( a\right) \right]
^{h\left( t\right) }\left[ f\left( b\right) \right] ^{h\left( 1-t\right) } \\
g\left( ta+\left( 1-t\right) b\right) &\leq &\left[ g\left( a\right) \right]
^{h\left( t\right) }\left[ g\left( b\right) \right] ^{h\left( 1-t\right) }.
\end{eqnarray*}%
If we multiply both sides we get%
\begin{equation*}
f\left( ta+\left( 1-t\right) b\right) g\left( ta+\left( 1-t\right) b\right)
\leq \left[ f\left( a\right) g\left( a\right) \right] ^{h\left( t\right) }%
\left[ f\left( b\right) g\left( b\right) \right] ^{h\left( 1-t\right) }.
\end{equation*}%
By integrating the inequality from $0$ to $1$ over $t,$ and change the
variable $x=ta+\left( 1-t\right) b$ we have%
\begin{equation*}
\frac{1}{b-a}\int_{a}^{b}\left( fg\right) \left( x\right) dx\leq \int_{0}^{1}%
\left[ f\left( a\right) g\left( a\right) \right] ^{h\left( t\right) }\left[
f\left( b\right) g\left( b\right) \right] ^{h\left( 1-t\right) }dt.
\end{equation*}%
Since $h$ is symmetric about $\frac{1}{2}$ we have $h\left( t\right)
=h\left( 1-t\right) .$ So we have%
\begin{equation*}
\frac{1}{b-a}\int_{a}^{b}\left( fg\right) \left( x\right) dx\leq \int_{0}^{1}%
\left[ \left( fg\right) \left( a\right) \left( fg\right) \left( b\right) %
\right] ^{h\left( t\right) }dt.
\end{equation*}
\end{proof}

\begin{theorem}
Let $f$ and $g$ are $h$-$\log $-convex functions. For $\alpha ,\beta >0$ and 
$\alpha +\beta =1$ we have%
\begin{equation}
\frac{1}{b-a}\int_{a}^{b}\left( fg\right) \left( x\right) dx\leq \int_{0}^{1}%
\left[ \alpha \left\{ \left[ f\left( a\right) \right] ^{h\left( t\right) }%
\left[ f\left( b\right) \right] ^{h\left( 1-t\right) }\right\} ^{\frac{1}{%
\alpha }}+\beta \left\{ \left[ g\left( a\right) \right] ^{h\left( t\right) }%
\left[ g\left( b\right) \right] ^{h\left( 1-t\right) }\right\} ^{\frac{1}{%
\beta }}\right] dt  \label{d}
\end{equation}%
and%
\begin{equation}
\frac{1}{b-a}\int_{a}^{b}\left( fg\right) \left( x\right) dx\leq
\int_{0}^{1}\left\{ \alpha \left[ f\left( a\right) g\left( a\right) \right]
^{\frac{h\left( t\right) }{\alpha }}+\beta \left[ f\left( b\right) g\left(
b\right) \right] ^{\frac{h\left( 1-t\right) }{\beta }}\right\} dt.  \label{e}
\end{equation}
\end{theorem}

\begin{proof}
Since $f$ and $g$ are $h$-$\log $-convex functions we have%
\begin{eqnarray}
f\left( ta+\left( 1-t\right) b\right) &\leq &\left[ f\left( a\right) \right]
^{h\left( t\right) }\left[ f\left( b\right) \right] ^{h\left( 1-t\right) }
\label{f} \\
g\left( ta+\left( 1-t\right) b\right) &\leq &\left[ g\left( a\right) \right]
^{h\left( t\right) }\left[ g\left( b\right) \right] ^{h\left( 1-t\right) }. 
\notag
\end{eqnarray}%
If we multiply both sides and use the fact that $cd\leq \alpha c^{\frac{1}{%
\alpha }}+\beta d^{\frac{1}{\beta }}$ (for $\alpha ,\beta >0,$ $\alpha
+\beta =1$) we get%
\begin{equation*}
\left( fg\right) \left( ta+\left( 1-t\right) b\right) \leq \alpha \left\{ 
\left[ f\left( a\right) \right] ^{h\left( t\right) }\left[ f\left( b\right) %
\right] ^{h\left( 1-t\right) }\right\} ^{\frac{1}{\alpha }}+\beta \left\{ %
\left[ g\left( a\right) \right] ^{h\left( t\right) }\left[ g\left( b\right) %
\right] ^{h\left( 1-t\right) }\right\} ^{\frac{1}{\beta }}.
\end{equation*}%
By integrating the above inequality, we get the proof of (\ref{d}).

On the other hand after multiplying both sides of (\ref{f}) we can write%
\begin{equation*}
\left( fg\right) \left( ta+\left( 1-t\right) b\right) \leq \alpha \left[
f\left( a\right) g\left( a\right) \right] ^{\frac{h\left( t\right) }{\alpha }%
}+\beta \left[ f\left( b\right) g\left( b\right) \right] ^{\frac{h\left(
1-t\right) }{\beta }}.
\end{equation*}%
Then, by integrating the last inequality we get the proof of \ref{e}.
\end{proof}

\begin{theorem}
Let $f$ be an $h$-$\log $-convex function on $\left[ a,b\right] .$ For $%
\alpha ,\beta >0,$ $\alpha +\beta =1$ we have%
\begin{equation*}
f\left( \frac{a+b}{2}\right) \leq \alpha \frac{1}{b-a}\int_{a}^{b}f\left(
x\right) ^{\frac{h\left( \frac{1}{2}\right) }{\alpha }}dx+\beta \frac{1}{b-a}%
\int_{a}^{b}f\left( x\right) ^{\frac{h\left( \frac{1}{2}\right) }{\beta }}dx.
\end{equation*}
\end{theorem}

\begin{proof}
If we choose $t=\frac{1}{2}$ on Definition \ref{d1} we have%
\begin{equation*}
f\left( \frac{x+y}{2}\right) \leq \left[ f\left( x\right) f\left( y\right) %
\right] ^{h\left( \frac{1}{2}\right) }
\end{equation*}%
If we change the variable $x=ta+\left( 1-t\right) b$ and $y=\left(
1-t\right) a+tb$ we get%
\begin{equation*}
f\left( \frac{a+b}{2}\right) \leq f\left( ta+\left( 1-t\right) b\right)
^{h\left( \frac{1}{2}\right) }f\left( \left( 1-t\right) a+tb\right)
^{h\left( \frac{1}{2}\right) }
\end{equation*}%
If we use the inequality $cd\leq \alpha c^{\frac{1}{\alpha }}+\beta d^{\frac{%
1}{\beta }}$ (for $\alpha ,\beta >0,$ $\alpha +\beta =1$) we get%
\begin{equation*}
f\left( \frac{a+b}{2}\right) \leq \alpha f\left( ta+\left( 1-t\right)
b\right) ^{\frac{h\left( \frac{1}{2}\right) }{\alpha }}+\beta f\left( \left(
1-t\right) a+tb\right) ^{\frac{h\left( \frac{1}{2}\right) }{\beta }}.
\end{equation*}%
By integrating the last inequality over $t$ on $\left[ 0,1\right] $ we have%
\begin{equation*}
f\left( \frac{a+b}{2}\right) \leq \alpha \int_{0}^{1}f\left( ta+\left(
1-t\right) b\right) ^{\frac{h\left( \frac{1}{2}\right) }{\alpha }}dt+\beta
\int_{0}^{1}f\left( \left( 1-t\right) a+tb\right) ^{\frac{h\left( \frac{1}{2}%
\right) }{\beta }}dt.
\end{equation*}%
By rewriting the inequality by using suitable variable changings we get the
desired result.
\end{proof}

\section{\textbf{Results for }$h$-geometrically \textbf{Convex Functions}}

\begin{definition}
\label{d2}A positive function $f$ is called $h$-geometrically convex on a
real interval $I=\left[ a,b\right] ,$ if for all $x,y\in I$ and $t\in \left[
0,1\right] ,$%
\begin{equation}
f\left( x^{t}y^{\left( 1-t\right) }\right) \leq \left[ f\left( x\right) %
\right] ^{h\left( t\right) }\left[ f\left( y\right) \right] ^{h\left(
1-t\right) }  \label{d22}
\end{equation}%
where $h\left( t\right) $\ is a nonnegative function on $J,$ with $%
h:J\subseteq 
\mathbb{R}
\rightarrow 
\mathbb{R}
.$
\end{definition}

If $f$ is a positive $h$-geometrically concave function, then inequality is
reversed.

\begin{remark}
It is clear that when $h\left( t\right) =t$ in Definition \ref{d2}, $h$%
-geometrically convex (concave) become ordinary geometrically convex
(concave) function, and if we take $h\left( t\right) =t^{s}$ in Definition %
\ref{d2}, $h$-geometrically convex (concave) become $s$-geometrically convex
(concave) function.
\end{remark}

\begin{remark}
As we can write%
\begin{equation*}
x^{t}y^{\left( 1-t\right) }\leq tx+\left( 1-t\right) y
\end{equation*}%
for $t\in \left[ 0,1\right] $ and $x,y>0,$ we get all Theorems and
Corollaries given at Section 2\textbf{\ }for decreasing $h-$geometrically
convex functions.
\end{remark}

\begin{theorem}
Let $f$ be an $h$-geometrically convex function on $I.$ For every $x,y\in I$
with $x<y$ we get%
\begin{equation*}
\frac{1}{\ln y-\ln x}\int_{x}^{y}f\left( \gamma \right) f\left( \frac{xy}{%
\gamma }\right) \frac{d\gamma }{\gamma }\leq \int_{0}^{1}\left[ f\left(
x\right) f\left( y\right) \right] ^{h\left( t\right) +h\left( 1-t\right) }dt.
\end{equation*}
\end{theorem}

\begin{proof}
Since we choose $f$ is an $h$-geometrically convex function on $I$, we can
write%
\begin{eqnarray*}
f\left( x^{t}y^{1-t}\right) &\leq &\left[ f\left( x\right) \right] ^{h\left(
t\right) }\left[ f\left( y\right) \right] ^{h\left( 1-t\right) } \\
f\left( x^{1-t}y^{t}\right) &\leq &\left[ f\left( x\right) \right] ^{h\left(
1-t\right) }\left[ f\left( y\right) \right] ^{h\left( t\right) }.
\end{eqnarray*}%
If we multiply both sides of inequalities we get%
\begin{equation*}
f\left( x^{t}y^{1-t}\right) f\left( x^{1-t}y^{t}\right) \leq \left[ f\left(
x\right) f\left( y\right) \right] ^{h\left( t\right) +h\left( 1-t\right) }
\end{equation*}%
By integrating both sides respect to $t$ over $\left[ 0,1\right] $ we have%
\begin{equation*}
\int_{0}^{1}f\left( x^{t}y^{1-t}\right) f\left( x^{1-t}y^{t}\right) dt\leq
\int_{0}^{1}\left[ f\left( x\right) f\left( y\right) \right] ^{h\left(
t\right) +h\left( 1-t\right) }dt
\end{equation*}%
If we change the variable $\gamma =x^{t}y^{1-t},$ we get the desired result.
\end{proof}

\begin{theorem}
Let $f$ and $g$ are $h$-geometricially convex functions on $I.$ For $p>1,$ $%
\frac{1}{p}+\frac{1}{q}=1$ we get%
\begin{eqnarray*}
\int_{0}^{1}f\left( x^{t}y^{1-t}\right) g\left( x^{1-t}y^{t}\right) dt &\leq
&\left( \int_{0}^{1}f\left( x\right) ^{p^{2}h\left( t\right) }dt\right) ^{%
\frac{1}{p^{2}}}\left( \int_{0}^{1}g\left( y\right) ^{pqh\left( t\right)
}dt\right) ^{\frac{1}{pq}} \\
&&\times \left( \int_{0}^{1}f\left( y\right) ^{pqh\left( 1-t\right)
}dt\right) ^{\frac{1}{pq}}\left( \int_{0}^{1}g\left( x\right) ^{q^{2}h\left(
1-t\right) }dt\right) ^{\frac{1}{q^{2}}}
\end{eqnarray*}%
for every $x,y\in I$ with $x<y$ and $t\in \left[ 0,1\right] .$
\end{theorem}

\begin{proof}
As we choose $f$ and $g$ are $h$-geometricially convex functions on $I$ we
can write%
\begin{eqnarray*}
f\left( x^{t}y^{1-t}\right) &\leq &\left[ f\left( x\right) \right] ^{h\left(
t\right) }\left[ f\left( y\right) \right] ^{h\left( 1-t\right) } \\
g\left( x^{1-t}y^{t}\right) &\leq &\left[ g\left( x\right) \right] ^{h\left(
1-t\right) }\left[ g\left( y\right) \right] ^{h\left( t\right) }.
\end{eqnarray*}%
By multiplying both sides and integrate respect to $t$ over $\left[ 0,1%
\right] $ we have%
\begin{equation*}
\int_{0}^{1}f\left( x^{t}y^{1-t}\right) g\left( x^{1-t}y^{t}\right) dt\leq
\int_{0}^{1}\left[ f\left( x\right) g\left( y\right) \right] ^{h\left(
t\right) }\left[ f\left( y\right) g\left( x\right) \right] ^{h\left(
1-t\right) }dt.
\end{equation*}%
If we apply H\"{o}lder's inequality for $p>1,$ $\frac{1}{p}+\frac{1}{q}=1$
we get%
\begin{equation*}
\int_{0}^{1}f\left( x^{t}y^{1-t}\right) g\left( x^{1-t}y^{t}\right) dt\leq
\left( \int_{0}^{1}\left[ f\left( x\right) g\left( y\right) \right]
^{ph\left( t\right) }dt\right) ^{\frac{1}{p}}\left( \int_{0}^{1}\left[
f\left( y\right) g\left( x\right) \right] ^{qh\left( 1-t\right) }dt\right) ^{%
\frac{1}{q}}.
\end{equation*}%
Then by applying H\"{o}lder's inequality again, we get%
\begin{eqnarray*}
\int_{0}^{1}f\left( x^{t}y^{1-t}\right) g\left( x^{1-t}y^{t}\right) dt &\leq
&\left[ \left( \int_{0}^{1}f\left( x\right) ^{p^{2}h\left( t\right)
}dt\right) ^{\frac{1}{p}}\left( \int_{0}^{1}g\left( y\right) ^{pqh\left(
t\right) }dt\right) ^{\frac{1}{q}}\right] ^{\frac{1}{p}} \\
&&\times \left[ \left( \int_{0}^{1}f\left( y\right) ^{pqh\left( 1-t\right)
}dt\right) ^{\frac{1}{p}}\left( \int_{0}^{1}g\left( x\right) ^{q^{2}h\left(
1-t\right) }dt\right) ^{\frac{1}{q}}\right] ^{\frac{1}{q}}.
\end{eqnarray*}%
By rearranging the inequality, we get the desired rusult.
\end{proof}

\section{\protect\bigskip $h$-multi Convex Functions}

Finally, we can introduce the following definition.

\begin{definition}
\label{dd}A positive function $f$ is called $h$-multi convex on a real
interval $I=\left[ a,b\right] ,$ if for all $x,y\in I$ and $t,\lambda \in %
\left[ 0,1\right] ,$%
\begin{equation}
\lambda f\left( x^{t}y^{\left( 1-t\right) }\right) +\left( 1-\lambda \right)
f\left( tx+\left( 1-t\right) y\right) \leq \left[ f\left( x\right) \right]
^{h\left( t\right) }\left[ f\left( y\right) \right] ^{h\left( 1-t\right) }
\end{equation}%
where $h\left( t\right) $\ is a nonnegative function on $J,$ with $%
h:J\subseteq 
\mathbb{R}
\rightarrow 
\mathbb{R}
.$
\end{definition}

If $f$ is a positive $h$-multi concave function, then inequality is reversed.

\begin{remark}
It is clear that when $\lambda =0$ in Definition \ref{dd}, $h$-multi convex
(concave) become $h$-logarithmically convex (concave) function, and if we
take $\lambda =1$ in Definition \ref{dd}, $h$-multi convex (concave) become $%
s$-geometrically convex (concave) function.\bigskip 
\end{remark}

\begin{theorem}
Let $f$ be an $h-$multi convex function on $I.$ Then we get%
\begin{equation*}
\frac{1}{2}\left[ \frac{1}{\ln y-\ln x}\int_{x}^{y}\frac{f\left( \gamma
\right) }{\gamma }d\gamma +\frac{1}{y-x}\int_{x}^{y}f\left( \gamma \right)
d\gamma \right] \leq \int_{0}^{1}\left[ f\left( x\right) \right] ^{h\left(
t\right) }\left[ f\left( y\right) \right] ^{h\left( 1-t\right) }dt
\end{equation*}%
for all $x,y\in I$ and $t,\lambda \in \left[ 0,1\right] .$
\end{theorem}

\begin{proof}
From Definiton \ref{dd} we have%
\begin{equation*}
\lambda f\left( x^{t}y^{\left( 1-t\right) }\right) +\left( 1-\lambda \right)
f\left( tx+\left( 1-t\right) y\right) \leq \left[ f\left( x\right) \right]
^{h\left( t\right) }\left[ f\left( y\right) \right] ^{h\left( 1-t\right) }.
\end{equation*}%
If we integrate the inequality respect to $\lambda $ over $\left[ 0,1\right] 
$ we have%
\begin{equation*}
\int_{0}^{1}\left( \lambda f\left( x^{t}y^{\left( 1-t\right) }\right)
+\left( 1-\lambda \right) f\left( tx+\left( 1-t\right) y\right) \right)
d\lambda \leq \int_{0}^{1}\left[ f\left( x\right) \right] ^{h\left( t\right)
}\left[ f\left( y\right) \right] ^{h\left( 1-t\right) }d\lambda .
\end{equation*}%
So we have%
\begin{equation*}
\frac{f\left( x^{t}y^{\left( 1-t\right) }\right) +f\left( tx+\left(
1-t\right) y\right) }{2}\leq \left[ f\left( x\right) \right] ^{h\left(
t\right) }\left[ f\left( y\right) \right] ^{h\left( 1-t\right) }.
\end{equation*}%
Then by integrating the inequality respect to $t$ over $\left[ 0,1\right] $
we get the desired result.
\end{proof}

\end{document}